\documentclass[12pt,a4paper]{amsart}
\usepackage{amssymb}
\usepackage{amsmath}
\usepackage{amsfonts}
\usepackage{color}

\theoremstyle{plain}
\newtheorem{thm}{Theorem}[section]
\newtheorem{prop}[thm]{Proposition}
\newtheorem{cor}[thm]{Corollary}
\newtheorem{lem}[thm]{Lemma}

\newtheorem{theorem*}{Theorem}[]

\theoremstyle{definition}
\newtheorem{defn}[thm]{Definition}

\newtheorem{exmp}[thm]{Example}
\newtheorem{alg}[thm]{Algorithm}

\theoremstyle{remark}
\newtheorem{rem}[thm]{Remark}

\newcommand{\prodscal}[2]{\left<#1 ,#2\right>}

\newcommand{\C}{\mathbb{C}}
\newcommand{\N}{\mathbb{N}}
\newcommand{\R}{\mathbb{R}}

\newcommand{\val}{v}

\begin{document}


\title{Validity proof of Lazard's method\\
       for CAD construction}

\author{Scott McCallum}
\address{Department of Computing, Macquarie University,
         NSW 2109, Australia}

\author{Adam Parusi\'nski}
\address {Univ. Nice Sophia Antipolis, CNRS,  LJAD, UMR 7351, 06108 Nice, France}

\author{Laurentiu Paunescu}
\address{School of Mathematics and Statistics, University of Sydney,
         NSW 2006, Australia}



\begin{abstract}
In 1994 Lazard proposed an improved method for cylindrical algebraic decomposition (CAD). 
The method comprised a simplified projection operation together with a generalized cell lifting 
(that is, stack construction) technique.  
For the proof of the method's validity 
Lazard introduced a new notion of valuation of a multivariate polynomial at a point.  
However a gap in one of the key supporting results for his proof was subsequently noticed.
In the present paper we provide a complete validity proof of Lazard's method. 
Our proof is based on the classical parametrized version of Puiseux's theorem and basic 
properties of Lazard's valuation.  
This result is significant because Lazard's method can be applied to any finite family 
of polynomials, without any assumption 
on the system of coordinates. It therefore 
has wider applicability and may be more efficient
than other projection and lifting schemes for CAD.
\end{abstract}

\maketitle

\section{Introduction}
{\em Cylindrical algebraic decomposition (CAD)} of Euclidean $n$-space $\mathbb{R}^n$,
relative to a given set of $n$-variate integral polynomials $A$,
is an important tool in computational algebra and geometry.
It was introduced by Collins \cite{Collins:75} as the key component of a
method for quantifier elimination (QE) for real closed fields.
Applications of CAD and other QE techniques
include robot motion planning \cite{Schwartz_Sharir:83},
stability analysis of differential equations \cite{Hong_Liska_Steinberg:97},
simulation and optimization \cite{Weispfenning:97},
epidemic modelling \cite{BKNW:06}, and
programming with complex functions \cite{DBEW:12}.
A key operation for CAD construction is {\em projection}: the projection of a set
of $n$-variate integral polynomials is a set of $(n-1)$-variate integral
polynomials. In view of its central role, much effort has been devoted to improving
this operation \cite{McCallum:88, McCallum:98, Hong:90, Brown:01}.
Cell lifting, or stack construction, is also an important component of CAD.

In 1994 Lazard \cite{Lazard:94} proposed an improved method for CAD computation.
The method comprised a simplified projection operation together with
a generalized cell lifting process.
 However a gap in one of the key supporting results of \cite%
{Lazard:94} was subsequently noticed \cite{Collins:98, Brown:01}. 
This was disappointing because Lazard's proposed approach has some advantages
over other methods.  In particular, it is relatively simple and it requires no assumption 
on the system of coordinates.

Inherent in \cite{Lazard:94} is a certain notion of valuation of a multivariate polynomial and, more generally, multivariate  Laurent-Puiseux 
(that is fractional meromorphic) series, at a point.  The related notion of the
valuation-invariance of such series in a
subset of $\mathbb{R}^n$ is also implicit in \cite{Lazard:94}. Lazard's proposed approach is
in contrast with the classical approach 
based on the concept of the order (of vanishing) of a multivariate
polynomial or analytic function at a point, and the related concept of
order-invariance, see McCallum \cite{McCallum:88, McCallum:98}.

A partial validity proof of Lazard's projection only was recently published
by McCallum and Hong, \cite{McCallum_Hong:16}. 
It was shown there that
Lazard's projection is valid for CAD construction for so-called
well-oriented polynomial sets. The key underlying results related to
order-invariance rather than valuation-invariance, and the validity proof was
built upon established results concerning improved projection. While this was
an important step forward it was only a partial validation of
Lazard's approach since the method was not proved to work for non
well-oriented polynomials and it did not involve valuation-invariance.

The present paper
provides a complete validity proof of Lazard's method using his notion of valuation.
There is no restriction of the method to well-oriented sets.
This result is significant because Lazard's method has wider applicability
and may be more efficient
than other projection and lifting schemes for CAD.
Moreover, we are hopeful that the use of Lazard's projection in CAD construction
may permit greater exploitation of equational constraints, when present,
in further reducing projection sets in CAD based QE \cite{McCallum:99, McCallum:01}.

This paper is organised as follows.
We first recall Lazard's method and main
claim (Section 2). We then study the concept of Lazard's valuation (Section 3).  
In this paper we only consider
Lazard's valuation for a multivariate polynomial. 
Section 4 contains the statement of a key mathematical result 
(the Puiseux with parameter theorem) underlying
our validation of Lazard's method.
In Section 5 we present our proof of Lazard's main claim using
Lazard's notion of valuation. The main idea of the 
proof is to use monomial test curves that allow us to change the valuation 
invariance along an analytic submanifold to the order invariance.  
In the appendix at the end of the paper we present, 
for the reader convenience, a proof of the Puiseux with parameter theorem.

\section{Lazard's proposed method and claims}
\label{lazard}

Background material on CAD, and in particular its projection operation,
can be found in \cite{Arnon_Collins_McCallum:84, Collins:75, Collins:98,
Collins_Hong:91, Hong:90, McCallum:88, McCallum:98}.
We present a precise definition of the projection operator $P_L$
for CAD introduced by Lazard \cite{Lazard:94}.
Put $R_0 = \mathbb{Z}$ and, for $n \ge 1$, put
$R_n = R_{n-1}[x_n] = \mathbb{Z}[x_1, \ldots, x_n]$.
Elements of the ring $R_n$ will usually be considered
to be polynomials in $x_n$ over $R_{n-1}$.
We shall call a subset $A$ of $R_n$
whose elements are irreducible polynomials of positive degree
and pairwise relatively prime an {\em irreducible basis}.
(This concept is analogous to that of \emph{squarefree basis}
which is used in the CAD literature, for example \cite{McCallum:88}.)

\begin{defn}
[Lazard projection]
Let $A$ be a finite irreducible basis in $R_n$, with $n \ge 2$. 
The \emph{Lazard projection} $P_{L}(A)$ of $A$ 
is the subset of $R_{n-1}$ comprising the following polynomials:
\begin{enumerate}
  \item all leading coefficients of the elements of $A$,
  \item all trailing coefficients (i.e.  coefficients independent of
$x_{n}$) of the elements of $A$,
  \item all discriminants of the elements of $A$, and
  \item all resultants of pairs of distinct elements of $A$.
\end{enumerate}
\end{defn}


\begin{rem}\label{remarkbasis} 
Lazard's projection could alternatively be defined for a (slightly modified)
squarefree basis $A$ in $R_n$, as in \cite{Lazard:94}.
We use an irreducible basis in our definition because experience
has shown that this likely leads to a more efficient CAD algorithm
in practice and on average.
\end{rem}

\begin{rem}\label{remarkproj}
Let $A$ be an irreducible basis. Lazard's projection
$P_L(A)$ is contained in and is usually strictly smaller than the McCallum
projection $P_M(A)$ \cite{McCallum:88, McCallum:98}.
Indeed $P_M(A)$ includes the ``middle coefficients'' (i.e. those coefficients
other than the leading and trailing ones)
of the elements of $A$, which $P_L(A)$ omits. 
In other respects these two projection operators
are the same. 
However $P_L(A)$ contains 
and is usually strictly larger than the Brown-McCallum projection
$P_{BM}(A)$ \cite{Brown:01}. 
Indeed $P_{BM}(A)$ omits the trailing
coefficients of the elements of $A$,
which $P_L(A)$ includes, but in other respects is the same as $P_L(A)$.
This remark notwithstanding, the Lazard projection is still of interest
because of certain limitations of the Brown-McCallum projection.
The two chief drawbacks of the projection $P_{BM}(A)$ are as follows.
First, the method of \cite{Brown:01} could fail in case $A$ is not well-oriented
\cite{McCallum:98, Brown:01}. Second, the method requires that any 0-dimensional
nullifying cells \cite{McCallum:98, Brown:01} in each dimension be identified and
added during CAD construction. These drawbacks are elaborated in \cite{Brown:01}.
\end{rem}

Lazard \cite{Lazard:94} outlined a claimed CAD algorithm for $A \subset R_n$
and $\mathbb{R}^n$ which uses the projection set $P_L(A)$.
The specification of his algorithm requires the following
 concept of his valuation:

\begin{defn}
[Lazard valuation]
Let $K$ be a field.
Let $n \ge 1$, $f \in K[x_1, \ldots, x_n]$ nonzero,
and $\alpha = (\alpha_1, \ldots, \alpha_n) \in K^{n}$.
The \emph{Lazard valuation} (\emph{valuation}, for short)
$v_\alpha(f)$ of $f$ at $\alpha$
is the element $ {\bf v} = (v_1, \ldots, v_n)$ 
of $\mathbb{N}^n$ least (with respect to $\le_{lex}$) 
such that $f$ expanded about $\alpha$ has a term 
\[c(x_1 - \alpha_1)^{v_1} \cdots (x_n - \alpha_n)^{v_n}\] with $c \neq 0$.
(Note that $\le_{lex}$ denotes the {\em lexicographic order} on
$\mathbb{N}^n$ -- see next section.)
\end{defn}

\begin{exmp}
Let $n = 1$. Then $v_\alpha(f)$ is the familiar order 
$\mathrm{ord}_\alpha(f)$ of $f \in K[x_1]$
at $\alpha \in K$. 
Thus, for instance, if $f(x_1) = x_1^2 - x_1^3$ then $v_0(f) = 2$
and $v_1(f) = 1$. As another example,
let $n = 2$ and $f(x_1, x_2) = x_1 x_2^2 + x_1^2 x_2 = x_1 x_2(x_2 + x_1)$. 
Then $v_{(0,0)}(f) = (1,2)$, 
$v_{(1,0)}(f) = (0,1)$, and $v_{(0,1)}(f) = (1,0)$.
\end{exmp}

The above defines $v_\alpha(f)$ for $f \in K[x_1, \ldots, x_n]$ nonzero
and $\alpha \in K^n$. Lazard
\cite{Lazard:94} actually defined $v_\alpha(f)$ 
for nonzero elements $f$ of the much
larger domain of all Laurent-Puiseux (that is, fractional meromorphic) series in
$x_1-\alpha_1, \ldots, x_n-\alpha_n$ over $K$.
In this sense the above is a more limited definition of valuation.
With $K$, $n$ and $f$ as in the above definition, and
$S \subset K^n$, we say $f$ is \emph{valuation-invariant} in $S$ if
the valuation of $f$ is the same at every point of $S$.
Some basic properties of this Lazard valuation, and the
associated notion of valuation-invariance, are presented in Section 3 below. 
Lazard's proposed CAD algorithm also uses
a technique
for ``evaluating'' a polynomial $f \in R_n$ 
at a sample point in $\mathbb{R}^{n-1}$.
This technique is described in slightly more general terms as follows:

\begin{defn}
[Lazard evaluation]
Let $K$ be a field which supports explicit arithmetic computation.
Let $n \ge 2$, take a nonzero element $f$ in $K[x_1, \ldots, x_n]$,
and let $\alpha = (\alpha_1, \ldots, \alpha_{n-1}) \in K^{n-1}$.
The \emph{Lazard evaluation} $f_{\alpha}(x_n)\in K[x_n]$ of $f$ at $\alpha$ is 
defined to be the result of the following process
(which determines also nonnegative integers $v_i$, with $1 \le i \le n-1$):
\begin{enumerate}
\item[] $f_{\alpha} \leftarrow f$
\item[] For $i \leftarrow 1$ to $n-1$ do
  \begin{enumerate}
  \item[] $v_{i} \leftarrow$ the greatest integer $v$ 
          such that $(x_{i}-\alpha_{i})^{v}~|~f_{\alpha}$
  \item[] $f_{\alpha} \leftarrow f_{\alpha}/(x_{i}-\alpha_{i})^{v_{i}}$
  \item[] $f_{\alpha} \leftarrow f_{\alpha}(\alpha_{i},x_{i+1},\ldots,x_{n})$
  \end{enumerate}
\end{enumerate}
\end{defn}

\begin{exmp}
We illustrate the above evaluation method using two simple examples.
For both examples we take $K = \mathbb{Q}$, $n = 3$ and $\alpha = (0,0)$.
We denote $(x_1, x_2, x_3)$ by $(x, y, z)$.
First let $f(x,y,z) = z^2 + y^2 + x^2 - 1$.
After the first pass through the method ($i = 1$) we have $v_1 = 0$ and
$f_{\alpha}(y,z) = z^2 + y^2 - 1$.
After the second pass ($i = 2$) we have $v_2 = 0$ and $f_{\alpha}(z) = z^2 - 1$.
In this case $f_{\alpha}(z) = f(0,0,z)$.
For our second example let $f(x,y,z) = yz - x$.
After the first pass ($i = 1$) we have $v_1 = 0$ and $f_{\alpha}(y,z) = yz$.
After the second pass ($i = 2$) we have $v_2 = 1$ and $f_{\alpha}(z) = yz/y = z$.
In this case $f_{\alpha}(z) \neq f(0,0,z)$, because the latter
polynomial is zero.
\end{exmp}

\begin{rem}
Notice that the assertion ``$f_{\alpha} \neq 0$'' is an invariant of the above process.
With $K$, $n$, $f$, $\alpha$ and the $v_i$
as in the above definition of Lazard evaluation,
notice that $f(\alpha, x_n) = 0$ (identically) if and only if $v_i > 0$, for some $i$
in the range $1 \le i \le n-1$.
With $\alpha_n \in K$ arbitrary,
notice also that the integers $v_i$, with $1 \le i \le n-1$,
are the first $n-1$ coordinates of $v_{(\alpha, \alpha_n)}(f)$.
It will on occasion be handy to refer to the $(n-1)$-tuple 
$(v_1, \ldots, v_{n-1})$ as the
\emph{Lazard valuation of} $f$ \emph{on} $\alpha$.
\end{rem}

\begin{rem}\label{secondremark}
Let   $f \in K[x_1, \ldots, x_n]$ be nonzero,  $\alpha = (\alpha_1, \ldots, \alpha_{n-1}) \in K^{n-1}$, and let  $\mathbf v= (v_1, \ldots, v_{n-1})$ be the Lazard valuation of $f$ on  $\alpha$.  
If we expand $f$ at $\alpha$ 
\begin{align}\label{evaluationexpansion}
f(x_1, \ldots, x_n) = \sum_{u} f^{\mathbf u} (x_n) \prod _{i=1}^{n-1} (x_i-\alpha_i)^{ u_i} 
\end{align}
where $\mathbf u=(u_1, \ldots ,u_{n-1})\in \N^{n-1}$, with coefficients  $f^{\mathbf u} (x_n) \in K[x_n]$, then $f_{\alpha} = f^ {\mathbf v}  $.   This follows from the fact that $\mathbf v$ is the minimum of 
$\{\mathbf u: f^{\mathbf u} \neq 0\}$ for the lexicographic order. 
\end{rem}

One more definition is needed before we can state Lazard's main claim
and his algorithm based on it. This definition is not explicit in
\cite{Lazard:94} -- it was introduced in \cite{McCallum_Hong:16}
to help clarify and highlight Lazard's main claim:

\begin{defn}\label{delineability}
[Lazard delineability]
With $K = \R$ and $x$ denoting $(x_1, \ldots, x_{n-1})$,
let $f$ be a nonzero element of $\R[x,x_n]$ and
$S$ a subset of $\mathbb{R}^{n-1}$.
We say that $f$ is \emph{Lazard delineable} on $S$ if
\begin{enumerate}
\item the Lazard valuation of $f$ on $\alpha$ is the same for each
point $\alpha \in S$;
\item 
there exist finitely many continuous functions
$\theta_1 < \cdots < \theta_k$ from $S$ to $\mathbb{R}$, with $k \ge 0$,
such that, for all $\alpha \in S$, the set of real roots
of $f_\alpha(x_n)$ is $\{\theta_1(\alpha), \ldots, \theta_k(\alpha)\}$
(where in case $k = 0$, this means that, for all $\alpha \in S$,
the set of real roots of $f_\alpha(x_n)$ is empty); 
and in case $k > 0$
\item there exist positive integers $m_1, \ldots, m_k$ such that,
for all $\alpha \in S$ and all $i$, $m_i$ is
the multiplicity of $\theta_i(\alpha)$ as a root of $f_{\alpha}(x_n)$.
\end{enumerate}
When $f$ is Lazard delineable on $S$ 
we refer to the graphs of the $\theta_i$ as the \emph{Lazard sections} of $f$
over $S$, and to the $m_i$ as the \emph{associated multiplicities} of these sections. 
The regions between successive Lazard sections,
together with the region below the lowest Lazard section
and that above the highest Lazard section, are called \emph{Lazard sectors}.
\end{defn}

\begin{prop}
If $f$ is Lazard delineable on $S$ then $f$ is valuation-invariant in every
Lazard section and sector of $f$ over $S$.
\end{prop}

\begin{proof}
Let $(v_1, \ldots, v_{n-1})$ be the common value of the Lazard valuation of $f$ on
$\alpha \in S$. Then $(v_1, \ldots, v_{n-1}, 0)$ is the Lazard valuation of $f$
at a point $(\alpha,z)$ in a Lazard sector of $f$ over $S$.
If $(\alpha, z)$ is in a Lazard section of $f$ over $S$
with associated multiplicity $m$ then
the Lazard valuation of $f$ at this point equals $(v_1, \ldots, v_{n-1}, m)$.
\end{proof}
%

We express Lazard's main claim, essentially the content of his Proposition 5
and subsequent remarks, as follows (as in \cite{McCallum_Hong:16}):

\begin{center}
\medskip
\parbox{11cm}{
Let $A$ be a finite irreducible basis in $R_n$, where $n \ge 2$.
Let $S$ be a connected subset of $\mathbb{R}^{n-1}$.
Suppose that each element of $P_L(A)$ is valuation-invariant in $S$.
Then each element of $A$ is Lazard delineable on $S$, and
the Lazard sections over $S$ of the elements of $A$ are pairwise disjoint.
}
\medskip
\end{center}

Our wording of this claim is different from Lazard's -- we have tried to
highlight and clarify the essence of his assertions.
This claim concerns valuation-invariant lifting in relation to $P_L(A)$:
it asserts that the condition, ``each element of $P_L(A)$ is valuation-invariant
in $S$'', is sufficient for an $A$-valuation-invariant stack
in $\mathbb{R}^n$ to exist over $S$. 
We prove Lazard's main claim in Section 5.
We can now describe Lazard's proposed CAD algorithm
(as in \cite{McCallum_Hong:16}):

\begin{alg}[Valuation-invariant CAD using Lazard projection]\mbox{}

\medskip
\noindent $(\mathcal{I}, \mathcal{S}) \leftarrow \mathrm{VCADL}(A)$

%

\medskip
\noindent \emph{Input}: $A$ is a list of integral polynomials in $x_{1},\ldots,x_{n}$. \\
\noindent \emph{Output}: $\mathcal{I}$ and $\mathcal{S}$ are lists of indices
and sample points, respectively, of the cells
comprising an $A$-valuation-invariant CAD of $\mathbb{R}^n$.

\medskip
\begin{enumerate}
\item If $n>1$ then go to (2).
\item[] Isolate the real roots of the irreducible factors 
        of the nonzero elements of $A$.
        (Algorithms for univariate integral polynomial factorization 
        are given in \cite{Kaltofen:82}.)
\item[] Construct cell indices $\mathcal{I}$ and 
        sample points $\mathcal{S}$ from the real roots. Exit. 
\item   $B \leftarrow$ the finest squarefree basis for $\mathrm{prim}(A)$.
        That is, $B$ is assigned the set of 
        ample irreducible factors of elements
        of the set $\mathrm{prim}(A)$ of primitive parts 
        of elements of $A$ of positive degree.
        (Recall that an ample set in a commutative ring with 1 is a set
         which contains exactly one element in each equivalence class of associates
         \cite{Collins:73, Collins:75}.
        Algorithms for multivariate integral polynomial factorization are given in
        \cite{Kaltofen:82}.)
\item[] $P \leftarrow \mathrm{cont}(A) \cup P_L(B)$.
        ($\mathrm{cont}(A)$ denotes the
        set of contents of elements of $A$.)
\item[] $(\mathcal{I}^{\prime}, \mathcal{S}^{\prime}) \leftarrow \mathrm{VCADL}(P)$.
\item[] $\mathcal{I}\leftarrow$ the empty list. 
        $\mathcal{S}\leftarrow$ the empty list. 
\item[] For each $\alpha = (\alpha_1, \ldots, \alpha_{n-1})$ 
                 in $\mathcal{S}^{\prime}$ do
  \begin{enumerate}
  \item[] Let $i$ be the index of the cell containing $\alpha$.
  \item[] $f^* \leftarrow$ $\prod_{f \in B} f_\alpha$. 
          (Each $f_\alpha$ is constructed using exact arithmetic in
          $\mathbb{Q}(\alpha)$.)
  \item[] Isolate the real roots of $f^*$.
  \item[] Construct cell indices and sample points for Lazard sections and sectors of
          elements of $B$ from $i$, $\alpha$ and the real roots of $f^*$.
  \item[] Add the cell indices to $\mathcal{I}$ and 
          the sample points to $\mathcal{S}$.
  \end{enumerate}
\item[] Exit.
\end{enumerate}
\end{alg}

As mentioned in Remark \ref{remarkbasis}, practical experience with CAD
suggests that it is worthwhile overall to first compute the finest squarefree basis
for $\mathrm{prim}(A)$ in step (2).
In particular, the cost in practice of such computation usually remains a relatively
small part of the total time for CAD.
But the benefit of such computation as the projection and lifting phases of CAD
proceed likely outweighs any additional cost of performing the factorization.
This is due in part to the expected presence of nontrivial factors of repeated
discriminants and resultants \cite{LM:09, BuseMourrain:09}.

The correctness of the above algorithm -- namely, the claim that,
given $A \subset R_n$, it produces a CAD of
$\mathbb{R}^n$ such that each cell of the CAD is valuation-invariant
with respect to each element of $A$ -- follows from
Lazard's main claim by induction on $n$.

\section{Basic properties of Lazard's valuation}

In this section we study Lazard's valuation \cite{Lazard:94}
in the relatively special setting, namely that of multivariate polynomials
over a ring, in which it was defined in the previous section. 
We shall clarify the notion and basic properties of this special valuation,
and identify some relationships between valuation-invariance and
order-invariance.
The content of this section is based on very similar material found in
\cite{McCallum_Hong:15}.

We recall the standard algebraic definition of the term
valuation \cite{Danilov, AtiyahMacDonald, ZariskiSamuel:60}. A mapping $v~:~R -
\{0\}~\rightarrow~\Gamma, $ 
$R$ a ring, into
a totally ordered abelian monoid (written additively) $\Gamma$ is said to be
a \emph{valuation} of $R$ if the following two conditions are satisfied:

\begin{enumerate}
\item $v(fg) = v(f) + v(g)$ for all $f$ and $g$;

\item $v(f + g) \ge \min\{v(f), v(g)\}$, for all $f$ and $g$ (with $f+g \neq
0$).
\end{enumerate}

Perhaps the simplest and most familiar example of a valuation in
algebraic geometry is the order of an $n$-variate polynomial over a field $K$
at a point $\alpha \in K^n$. That is, the mapping 
$\mathrm{ord}_{\alpha}~:~K[x_1, \ldots, x_n] - \{0\}~\rightarrow~\mathbb{N}$ 
defined by 
\begin{equation*}
\mathrm{ord}_{\alpha}(f) = \mbox{the order of}~f~\mbox{at}~\alpha
\end{equation*}
is a valuation of the ring $K[x_1, \ldots, x_n]$. 
The order of $f$ at $\alpha$ is also called ``the multiplicity of $f$ at $\alpha$''.

Let $n \ge 1$. Recall that the \emph{lexicographic order} $\le_{lex}$ 
on $\mathbb{N}^n$ is defined by 
$v = (v_1, \ldots, v_n) \le_{lex} (w_1, \ldots, w_n) = w$ 
if and only if either $v = w$ or there is some $i$, $1 \le i \le n$,
with $v_j = w_j$, for all $j$ in the range $1 \le j < i$, and $v_i < w_i$.
Then $\le_{lex}$ is an \emph{admissible} order on 
$\mathbb{N}^n$ in the sense of 
\cite{BWK:98}. Indeed $\mathbb{N}^n$, together with componentwise addition
and $\le_{lex}$, forms a totally ordered abelian monoid. 
The lexicographic order 
$\le_{lex}$ can be defined similarly on $\mathbb{Z}^n$, forming a totally
ordered abelian group.\newline

Recall the definition of $v_\alpha(f)$ for $f \in K[x_1, \ldots, x_n]$ nonzero
and $\alpha \in K^n$ from Section 2:
$v_\alpha(f)$ is the element $(v_1, \ldots, v_n)$ of $\mathbb{N}^n$ 
least (with respect to $\le_{lex}$) such
that $f$ expanded about $a$ has a term 
$c(x_1 - \alpha_1)^{v_1} \cdots (x_n - \alpha_n)^{v_n}$ with $c \neq 0$. 
Notice that $v_\alpha(f) = (0, \ldots, 0)$ if and only if $f(\alpha) \neq 0$. 
%
Where there is no ambiguity we shall usually omit the qualifier ``Lazard''
from ``Lazard valuation''. We state some basic properties 
of the valuation $v_\alpha(f)$, 
analogues of properties of the familiar order $\mathrm{ord}_\alpha(f)$.
The first property is the fulfilment  of the axioms.

\begin{prop}
Let $f$ and $g$ be nonzero elements of $K[x_1, \ldots, x_n]$ and let $\alpha \in
K^n$. Then $v_\alpha(fg) = v_\alpha(f) + v_\alpha(g)$ and 
$v_\alpha(f + g) \ge_{lex} \min\{v_\alpha(f), v_\alpha(g)\}$ (if $f + g \neq 0$).
\end{prop}

\begin{proof}
These claims follow since $\mathbb{N}^n$, together with componentwise
addition and $\le_{lex}$, forms a totally ordered abelian monoid.
\end{proof}

%
%

\begin{prop}\label{semicontinuity}
(Upper semicontinuity of valuation) Let $f$ be a nonzero element of $K[x_1,
\ldots, x_n]$ and let $v = (v_1, \ldots, v_n) \in \N^n$.  Then the set 
$\{ \gamma \in K^n; v_\gamma (f) \ge_{lex} v\} $ is an algebraic subset of  $K^n$.  
In particular, the Lazard valuation is upper semi-continuous (in Zariski topology for any 
field $K$ and in the classical topology for $K
= \mathbb{R}$ or $\mathbb{C}$).  
\end{prop}

\begin{proof}
Denote $w = (w_1, \ldots, w_n)$, $\alpha= (\alpha_1, \ldots , \alpha _n)$.  The coefficient $c_{w, \alpha} $  in the expansion of $f$ at $\alpha$ 
\[f= \sum _w c_{w,\alpha} (x_1 - \alpha_1)^{w_1} \cdots (x_n - \alpha_n)^{w_n},\]
for $w$ fixed, is a polynomial in $\alpha$. (If $K$ is of characteristic zero then this coefficient 
equals  $\displaystyle {
\frac 1 {w_1! \cdots w_n!} \frac { \partial^{w_1+ \cdots + w_n}f}{\partial x_1^{w_1} \cdots \partial x_n^{w_n}}}$.)
The set $\{ \alpha \in K^n; v_\alpha (f) \ge_{lex} v\} $ is the intersection of the zero set of polynomials   
$c_{w, \alpha} $ for 
$w = (w_1, \ldots, w_n) <_{lex} v$.  Therefore it is algebraic.   
The algebraic sets are closed in Zariski topology by definition and clearly also, 
if $K= \mathbb{R}$ or $\mathbb{C}$, in the classical topology.  
\end{proof}

\begin{rem}\label{remarkanalytic}
Let $f : U \rightarrow K$,  $K
= \mathbb{R}$ or $\mathbb{C}$, be analytic, where $U$ is an open connected subset of
$K^n$, and suppose that $f$ does not vanish identically.
Then the Lazard valuation $v_{\alpha}(f)$ for $\alpha \in U$ can be defined 
exactly as in Definition 2.4 
and it satisfies the upper semicontinutiy property for the classical topology.  But the Lazard 
valuation extended to rational or meromorphic functions does not satisfy the upper semicontinutiy, 
see \cite{McCallum_Hong:15} for a discussion.
\end{rem}

We shall say that $f$ is \emph{%
valuation-invariant} in a subset $S\subset K^n$ 
 if $v_\alpha(f)$ 
is constant as $\alpha$ varies in $S$.

\vspace{0.5cm} For the remaining properties we state we shall assume that 
$K = \mathbb{R}$ or $\mathbb{C}$.
(We will use Proposition \ref{valinvproduct} later but not the remaining
properties.)

\begin{prop}\label{valinvproduct}
Let $f$ and $g$ be analytic in $U \subset K^n$, with $U$ open and connected, 
and suppose neither $f$ nor $g$ vanishes identically. 
(In particular, $f$ and $g$ could be nonzero elements of $K[x_1, \ldots, x_n]$, with $U = K^n$.)
Let $S \subset U$ be connected. 
Then $fg$ is valuation-invariant in $S$ if and
only if both $f$ and $g$ are valuation-invariant in $S$.
\end{prop}

\begin{proof} This follows easily from Proposition \ref{semicontinuity}
and Remark \ref{remarkanalytic}.   See  for instance the proof of Lemma A.3 of \cite{McCallum:88}. 
\end{proof}


The next lemma is in a sense another
analogue of the familiar order, and is particular to the case $n = 2$.

\begin{lem}\label{twovariableslemma}
Let $f(x,y) \in K[x,y]$ be primitive of positive degree in $y$ and
squarefree. Then for all but a finite number of points $(\alpha, \beta) \in
K^2$ on the curve $f(x,y) = 0$ we have $v_{(\alpha,\beta)}(f) = (0,1)$.
\end{lem}

\begin{proof}
Denote by $R(x)$ the resultant $\mathrm{res}_y(f, f_y)$ of $f$ and $f_y$
with respect to $y$. Then $R(x) \neq 0$ since $f$ is assumed squarefree. Let 
$(\alpha, \beta) \in K^2$, suppose $f(\alpha, \beta) = 0$ and assume that $%
v_{(\alpha,\beta)}(f) \neq (0,1)$. Then $f_y(\alpha,\beta) = 0$. Hence $%
R(\alpha) = 0$. So $\alpha$ belongs to the set of roots of $R(x)$, a finite
set. Now $f(\alpha, \beta) = 0$ and $f(\alpha, y) \neq 0$, since $f$ is
assumed primitive. So $\beta$ belongs to the set of roots of $f(\alpha, y)$,
a finite set.
\end{proof}

Let us further consider the relationship between the concepts of order-invariance
and valuation-invariance for a subset $S$ of $K^n$. The concepts are the
same in case $n = 1$ because order and valuation are the same for this case.
For $n = 2$ order-invariance in $S$ does not imply valuation-invariance in $S
$. (For consider $S= \{(x,y)| x^2 + y^2 - 1=0\}$  in $K^2$. The order of 
$f(x,y) = x^2 + y^2 - 1$ at every point of $S$ is 1. The valuation of $f$ at
every point $(\alpha, \beta) \in S \, \, \text{except} \, (\pm 1,0)$ is $(0,1)$.
But $v_{(\pm 1,0)}(f) = (0,2)$.) However for $n = 2$
we can prove the following.

\begin{prop}\label{twovariablesprop}
Let $f \in K[x,y]$ be nonzero and $S \subset K^2$ be connected. If $f$ is
valuation-invariant in $S$ then $f$ is order-invariant in $S$.
\end{prop}

\begin{proof}
Assume that $f$ is valuation-invariant in $S$. Write $f$ as a product of
irreducible elements $f_i$ of $K[x,y]$. By Proposition \ref{valinvproduct} each $f_i$ is
valuation-invariant in $S$. We shall show that each $f_i$ is order-invariant
in $S$. Take an arbitrary factor $f_i$. If the valuation of $f_i$ in $S$ is $%
(0,0)$ then the order of $f_i$ throughout $S$ is $0$, hence $f_i$ is
order-invariant in $S$. So we may assume that the valuation of $f_i$ is
nonzero in $S$, that is, that $S$ is contained in the curve $f_i(x,y) = 0$.
Suppose first that $f_i$ has positive degree in $y$. Now the conclusion is
immediate in case $S$ is a singleton, so assume that $S$ is not a singleton.
Since $S$ is connected, $S$ is an infinite set. By Lemma \ref{twovariableslemma} and
valuation-invariance of $f_i$ in $S$, we must have $v_{(\alpha,\beta)}(f_i)
= (0,1)$ for all $(\alpha,\beta) \in S$. Hence $f_i$ is order-invariant in $S
$ (since $\mathrm{ord} f_i = 1$ in $S$). Suppose instead that $f_i = f_i(x)$
has degree 0 in $y$. Since $f_i(x)$ is irreducible it has no multiple roots.
Therefore $v_{(\alpha,\beta)}(f_i) = (1,0)$ for all $(\alpha,\beta) \in S$.
Hence $f_i$ is order-invariant in $S$ (since $\mathrm{ord} f_i = 1$ in $S$).
The proof that $f_i$ is order-invariant in $S$ is finished and
we conclude that 
$f$ is order-invariant in $S.$ 
\end{proof}

However the following example indicates that Proposition
\ref{twovariablesprop} is not true for dimension greater than 2; that is,
valuation-invariance does not imply
order-invariance when $n > 2$. Let $f(x,y,z) = z^2 - x y$ and let $S$ be the 
$x$-axis in $\mathbb{R}^3$. Now $f$ is valuation-invariant in $S$, since the
valuation of $f$ at each point of $S$ is $(0,0,2)$. But $f$ is not
order-invariant in $S$, 
since $\mathrm{ord}_{(0,0,0)} f = 2$ and 
$\mathrm{ord }_{(\alpha,0,0)} f = 1$ for $\alpha \neq 0$.


\section{The Puiseux with parameter theorem}
We recall the classical Puiseux with parameter theorem
in the form given in \cite{Pawlucki:84}.  
This theorem is a special case of the Abhyankar-Jung theorem, 
see \cite{Abhyankar:55}, \cite{PR:12}, and hence can be traced back to \cite{Jung:08}.
Puiseux with parameter is closely related to certain algebraic results
of Zariski concerning equisingularity in codimension 1 over an
arbitrary algebraically closed field of characteristic 0
(Thm 7 of \cite{ZariskiI:65} and Thms 4.4 and 4.5 of \cite{ZariskiII:65}).
In the Appendix at the end of this paper 
we provide a short proof of this theorem for the reader's convenience.   

We use the following notation:
with $\varepsilon = (\varepsilon_1, \ldots, \varepsilon_k)$, 
$U_{ \varepsilon, r} =  U_{\varepsilon} \times U_r  $, where 
$U_\varepsilon = \{x = (x_1, \ldots, x_k) \in \C^k:  |x_i|< \varepsilon_i , \forall i\}$, 
$U_r= \{y\in \C: |y|<r\}$.   In this section ``analytic'' means ``complex analytic''.

\begin{thm} \label{PuiseuxTheorem}{\rm (Puiseux with parameter)}\\
Let 
\begin{align}\label{polynomial2}
f(x,y,z) = z^d+ \sum_{i =0}^{d-1}a_i(x,y) z^{i} ,
\end{align}
be a monic polynomial in $z$ with coefficients $a_i(x,y)$ analytic in $U_{ \varepsilon, r}$.   
Suppose that the discriminant of $f$ is of the form  $D_f (x,y) = y^m u(x,y)$ 
with analytic function $u$ non vanishing on  $U_{ \varepsilon, r} $.  
Then, there are  a positive integer $N$ (we may take $N=d!$) 
and analytic functions $\xi_i(x,t): U_{\varepsilon,r^{1/N}}\to \C $ such that 
$$f(x,t^N,z) = \prod _{i=1}^d  (z - \xi_i (x,t))$$
for all $(x,t) \in U_{\varepsilon, r^{1/N}}$.  

The roots $\xi_i (x,t)$ satisfy, moreover, the following properties.  
Firstly, for every $i\ne j$,   
$\xi_i-\xi_j$ equals a power of $t$ times a function 
nonvanishing on  $U_{ \varepsilon, r^{1/N}} $.  Secondly, 
if the coefficients   $a_i(x,y)$ of $f$ are (complexifications of) real analytic functions then 
  the set of functions $\xi_i (x,t)$ is complex conjugation invariant. 
\end{thm}


Note that ``parameter'' in the name ``Puiseux with parameter'' refers to the
$k$-tuple $x$, which ``parametrizes'' the Puiseux roots $\xi_i(x,t)$.
Observe also that Theorem \ref{PuiseuxTheorem} is applicable to any nonmonic
polynomial $f(x,y,z)$ which otherwise satisfies the hypotheses of the theorem,
and for which the leading coefficient $a_d(x,y)$ vanishes nowhere
in $U_{\varepsilon,r}$. For we may simply divide $f$ by $a_d$
to obtain a monic polynomial with the same roots which still satisfies the
hypotheses of the theorem.

By analogy with \cite{McCallum_Hong:16} we now consider the more general case in which
\begin{align}\label{polynomial3}
f(x,y,z) = a_d(x,y) z^d+ \sum_{i =0}^{d-1}a_i(x,y) z^{i} , 
\end{align}
with coefficients $a_i(x,y)$ analytic in $U_{\varepsilon,r}$,  
under the assumption that each of $a_d (x,y)$ and $D_f (x,y)$ 
is equal to a power of $y$ times a unit, that is, an analytic function
nonvanishing in $U_{\varepsilon,r}$.
We may then apply Theorem \ref{PuiseuxTheorem} to $\tilde f(x,y,\tilde z)$ defined by 
\begin{align}\label{tildeF}
\tilde f(x,y,\tilde z) = \tilde z^d+ \sum_{i =0}^{d-1}a_i a_d^{d-1-i} \tilde z^{i} 
\end{align}
because $D_{\tilde f}  = a_d^{d^2 -3d+2} D_f$ 
satisfies the assumptions of Theorem \ref{PuiseuxTheorem}.  
Let $ \eta_i(x,t)$  be the roots of $\tilde f(x,t^N, \tilde z)$.  Then, with $y = t^N$,
\begin{align}\label{relations}
  \prod _{i=1}^d  (a_d z-   \eta_i )  =  \tilde f (x,y, a_d z)  
 = a_d ^{d-1} f(x,y,z) . 
  \end{align} 
The above relation is key for proving the following result, which is
a complex analytic version of Corollary 3.15 of \cite{McCallum_Hong:16}.

\begin{cor}\label{Puiseuxfinal}
Let $f(x,y,z) $ be as in \eqref{polynomial3} and 
suppose that each of $a_d (x,y)$, $a_0(x,y)$ and $D_f (x,y)$ 
is of the form a power of $y$ times an 
analytic function nonvanishing on $U_{\varepsilon,r}$.
Then there are an integer $N > 0$,  analytic functions 
$u_i(x,t): U_{\varepsilon,r^{1/N}}\to \C $, $1 \le i \le d$ and  
non-negative integers $m_0,m_1, \ldots, m_d$, such that 
on $U_{\varepsilon, r^{1/N}}$,
$$t^{dm_0} f(x,t^N,z) = a_d(x,t^N) \prod _{i=1}^d  (t^{m_0} z - t^{m _i} u_i (x,t)).$$


The functions $u_i (x,t)$ satisfy, moreover, the following properties. 
For all $i$, $u_i$ is nonvanishing in $U_{\varepsilon, r^{1/N}}$.
For every $i \ne j$, $t^{m_i} u_i- t^{m_j} u_j$ 
equals a power of $t$ times a function nonvanishing on  $U_{ \varepsilon, r^{1/N}} $.    
If the coefficients of  $a_i(x,y)$ of $f$ are (complexifications of) 
real analytic functions then 
the set of functions $u_i (x,t)$ is complex conjugation  invariant. 
\end{cor}

\begin{proof}
By  the assumption $a_d(x,t^N)=  t^{m_0} \tilde a_d(x,t^N) $ and 
$a_0(x,t^N)=  t^{m} \tilde a_0(x,t^N) $, with $\tilde a_d$, $\tilde a_0$ nowhere 
vanishing.  
Since 
$$\prod _i \eta_i (x,t^N) = a^{d-1} _d(x,t^N) a_0(x,t^N), $$ 
the same is true for each $\eta_i $ (the roots of  $\tilde f$), namely,
$\eta_i (x,t)= t^{m_i} \tilde \eta_i (x,t)$, with $ \tilde \eta_i (x,t)$ nowhere vanishing.  
By \eqref{relations} we get 
$ \tilde a_d^d \prod _{i=1}^d  (t^{m_0} z-   \tilde a^{-1}_d\eta_i )  =  
\tilde a_d^d \prod _{i=1}^d  (t^{m_0} z-   t^{m_i}\tilde a^{-1}_d\tilde\eta_i ) = 
a_d ^{d-1} f(x,y,z) $. 
Hence, setting $u_i : = \tilde a^{-1}_d(x,t^N)\tilde \eta_i (x,t)$  we obtain the required 
formula for $t^{dm_0}  f(x,t^N,z) $.   If the coefficients  $a_i$ are real then the set of roots $\eta_i$ is conjugation invariant and hence so is the set of functions $u_i (x,t)$
\end{proof}

\begin{rem}
By analogy with Corollary 3.15 of \cite{McCallum_Hong:16}, just the hypotheses
on the leading coefficient $a_d(x,y)$ and discriminant $D_f(x,y)$ of
a nonmonic polynomial $f(x,y,z)$ suffice
to yield a conclusion slightly weaker than that presented in Corollary 
\ref{Puiseuxfinal} above.
In particular, 
the existence of integer $N>0$, analytic functions $u_i(x,t)$ and nonnegative
integers $m_i$ yielding a factorization of
$t^{dm_0}f(x,t^N,z)$ follow from just these hypotheses.
Formally, the roots $t^{m_i - m_0} u_i(x,t)$ of $f(x,t^N,z)$ are meromorphic
(Laurent) series: this means that finitely many negative exponents in $t$,
as well as $u_i \equiv 0$, are allowed.
Hence some roots could go to infinity as $t \rightarrow 0$,
where such behaviour in general depends on the parameter $x$.
(With $k = 1$, consider the example $f(x,y,z) = yz - x$.)
The additional hypothesis of Corollary \ref{Puiseuxfinal} above
is that the trailing coefficient $a_0(x,y)$ is of the same special form
as $a_d(x,y)$ and $D_f(x,y)$.
This hypothesis on $a_0(x,y)$ is key to proving the \emph{nonvanishing} of
the functions $u_i(x,t)$ in $U_{\varepsilon,r^{1/N}}$, which is not guaranteed
without it (as the example shows).
The nonvanishing of the $u_i(x,t)$ implies, in turn, that the number of roots
which tend to infinity as $t \rightarrow 0$ is independent of $x$.
This consequence is crucial to part of the proof of our main theorem
presented in the next section.
\end{rem}

As follows from the next lemma 
the assumptions of Theorem \ref{PuiseuxTheorem} and Corollary \ref{Puiseuxfinal} 
can be expressed equivalently as the order invariance of  
$a_d (x,y)$, $a_0(x,y)$ and $D_f (x,y)$ in the hyperplane $y=0$.

\begin{lem}\label{alonghypersurface}
Recall that $x$ denotes $(x_1, \ldots, x_{k})$ in this section. 
Let $g(x,y)$ be analytic in a neighbourhood of the origin in $K^{k+1}$, $K=\R$ or $\C$, 
and suppose that $g$ does not vanish identically.   The following are equivalent:
\begin{enumerate}
\item The order of  $g(a, y)$ (as a function of $y$) 
      at $y=0$ equals $m$ for all fixed $a$ sufficiently small.
\item For some function $u$ analytic near the origin 
      with $u(0, 0) \neq 0$ we have
      $g(x, y) = y^m u(x, y)$ for all $(x, y)$ sufficiently small.
\item  $g$ is order-invariant in the hyperplane $y= 0$ near the origin and this order is equal to $m$. 
\end{enumerate}
\end{lem}

\begin{proof}
First we show that (1) implies (2). Suppose that $g(\alpha, y)$ is of order $m$ at $y=0$ for $\alpha$ 
within box $B$ about the origin.  
Expand $g$ about the origin as the following iterated series:
\[g(x, y) = g_0(x) + g_1(x) y + g_2(x) y^2 + \cdots .\]
Take $(a, 0) \in B$. By assumption $g_{i}(a) = 0$ for all $i < m$ and 
 $g_{m}(a) \neq 0$. Setting
\[u(x, t) = g_{m}(x) + g_{m+1}(x) y + g_{m+2}(x) y^ 2 + \cdots ,\]
we have $g(x, y) = y^{m} u(x, y)$, 
and $u(0, 0) \neq 0$,
as required.

Our proof above that (1) implies (2) could be easily adapted 
to show that (3) implies (2). 
That (2) implies (1) and (3) is straightforward. 
\end{proof}

\begin{rem}
One can adapt easily the above proof to show that the conditions (1)-(3)
also are equivalent to \emph{
\begin{enumerate}
\item [(4)]g is valuation-invariant in the hyperplane y = 0 near the origin,
and this valuation is equal to (0,m) (where 0 denotes n-1 zeros).
\end{enumerate}}
We do not use this result in this paper.
\end{rem}



\section{Proof of Lazard's main claim}
We shall need to sharpen slightly the definition of Lazard delineability
given in Section 2. 
First recall that a \emph{k-dimensional analytic submanifold}
of $\mathbb{R}^{n-1}$ is a nonempty subset $S$ which looks locally like
$\mathbb{R}^k$. That is, for each point $p$ of $S$, there is an analytic
coordinate system $(\hat{x}_1, \ldots, \hat{x}_{n-1})$ about $p$ such that
near $p$, $S$ is the intersection of the $n-1-k$ hyperplanes
$\hat{x}_{k+1} = 0, \ldots, \hat{x}_{n-1} = 0$ in the local coordinate
system \cite{McCallum:88, McCallum:98}. With this notation, we say that
$(\hat{x}_1, \ldots, \hat{x}_{k})$ are \emph{local coordinates on S near p}.
A function $\theta : S \rightarrow \mathbb{R}$ is said to be {\em analytic} if near every
point $p$ of $S$, there are local coordinates on $S$ with respect to
which $f$ is analytic (near the origin in $\mathbb{R}^k$).
We say that an $n$-variate real polynomial $f$ 
is {\em Lazard analytic delineable}
on a submanifold $S$ of $\mathbb{R}^{n-1}$ 
if conditions (1), (2) and (3) of Definition \ref{delineability} are satisfied, 
where the continuous functions
$\theta_1 < \ldots < \theta_k$ from $S$ to $\mathbb{R}$ are moreover analytic.
The major effort of this section is to prove the following result; the Lazard claim is then an easy consequence.

\begin{thm}\label{mainthm}
Let $f(x, x_n) \in \mathbb{R}[x, x_n]$ 
have positive degree $d$ in $x_n$, where $x = (x_1, \ldots, x_{n-1})$.
Let $D(x)$, $l(x)$ and $t(x)$ denote the discriminant, leading coefficient
and trailing coefficient (that is, the coefficient independent of $x_n$) of $f$,
respectively, and suppose that each of these polynomials is nonzero
(as an element of $\R[x]$).
Let $S$ be a connected analytic submanifold of $\mathbb{R}^{n-1}$
in which $D$, $l$ and $t$ are all valuation-invariant.
Then $f$ is Lazard analytic delineable on $S$,
hence $f$ is valuation-invariant in every Lazard section and sector over $S$.
Moreover, the same conclusion holds for the polynomial 
$f^*(x, x_n) = x_nf(x, x_n)$.
\end{thm}

A special form of Lazard's main claim, in which $S$ is assumed to be
a connected submanifold of $\mathbb{R}^{n-1}$, 
and each element of $A$, an irreducible basis,  is concluded to be
Lazard analytic delineable on $S$ etc., follows from the above theorem
by Proposition 3.4. 
The details are provided in Subsection 5.3 below.
This special form of Lazard's main claim is sufficient to validate
Lazard's CAD method, as outlined in Section 2.

Our task now is to prove Theorem 5.1. To do this
we first need to investigate transforming the valuation of $f$ at a point $p$
into the order of $f$ along a curve passing through $p$.

\subsection{Transforming the valuation}
For a pair of vectors $\mathbf c,  \mathbf v$ we denote by $\prodscal {\mathbf c} {\mathbf v}$
 their scalar product  $\prodscal {\mathbf c} {\mathbf v} = \sum _{i} c_i v_i$.  Let 
 $V\subset \N^ n$ be a non-empty family. 
We say that an $n$-tuple of positive integers  
 $\mathbf c= (c_1, \ldots , c_n)\in (\N^*) ^ n$ is an
\emph{evaluator for} $V$ if for every  
$i= 1, ...,n-1$ we have
\begin{align}\label{conditiononci}
c_i \ge 1+  \max_{\mathbf v\in V} \sum _{j>i} c_j v_{j}  .
\end{align}
The above inequality is intended to mean that the subset
$\{ \sum_{j>i} c_jv_j | \mathbf v \in V\}$ of $\N$ is finite and, if so, $c_i$ strictly exceeds
the maximum element of this set.
Such a $\mathbf c$ always exists for any given 
finite  $V\subset \N^ n$, but may or may not exist if $V$ is infinite. 
Indeed, if $V$ is finite, we may choose $c_n$ arbitrarily and then use \eqref{conditiononci} to define 
$c_{n-1}, c_{n-2}, \ldots $ recursively. 
This is enough for us because we shall typically need an evaluator for a set $V$
of valuations of some $n$-variate polynomial $g$.  Then 
$\{v_\alpha(g)| \alpha \in K^n\}$ is finite, in fact every i-th  component $(v_\alpha(g))_i$ 
of $v_\alpha(g)$ is  bounded by the maximal exponent of $x_i^{\beta_i}$ that  appears in $g$.

\begin{rem}\label{herval}
If  $\mathbf c= (c_1, \ldots , c_n)$ is an evaluator for $V$ then 
${\mathbf c'} = (c_1, \ldots  c_{n-1})$ is an evaluator for 
$V' = \{(v_1,\ldots , v_{n-1})| \exists_{ v_n}  (v_1,\ldots , v_n)\in V\}$.
\end{rem}

\begin{lem}\label{testexponents}
Let $V\subset \N^ n$ be a non-empty family 
and let  $\mathbf c= (c_1, ..., c_n)\in (\N^*) ^ n$ be an
evaluator for $V$.  

If $\mathbf v \in V$, $\mathbf u \in \N^ n$, and 
$\mathbf v <_{lex} \mathbf u$ for the lexicographic order, 
then  $\prodscal {\mathbf c} {\mathbf v} < \prodscal {\mathbf c} {\mathbf u}$.  
In particular, if $\mathbf v, \mathbf u \in V$ 
then $\prodscal {\mathbf c} {\mathbf v} = 
\prodscal {\mathbf c} {\mathbf u}$ if and only if $\mathbf v = \mathbf u $. 
\end{lem}

\begin{proof}
Suppose that  $\mathbf v <_{lex} \mathbf u$, 
that is there is $i$ such that $u_k=v_k$ for $k<i$ and  $v_i<u_i$.   Then 
\begin{align*}
\prodscal {\mathbf c} {\mathbf u} - \prodscal {\mathbf c} {\mathbf v} 
& = c_i(u_i-v_i) + \sum _{j>i}   c_j (u_j-v_j)\\
& \ge (u_i-v_i) +  (c_i - 1)(u_i - v_i - 1) + [c_i - 1 - \sum _{j>i}  c_j v_j] > 0. 
\end{align*}
(The strict inequality uses the fact that the term in brackets is nonnegative
since $\mathbf c$ is an evaluator for $V$ and $\mathbf v \in V$.)
\end{proof}

The following result allows us to transform the Lazard valuation $\val_ p (f)$ 
into the order of $f$ at $y = 0$ along a monomial curve of the form  
$K\ni y\to p +  ( y^{c_1}, \ldots ,  y^{c_{n}}) $.

\begin{prop}\label{order}
Let $f\in K [x_1,...,x_n]$, $f\ne 0$, $p  \in K^n$.
With $V \subset \N^n$ a non-empty family,
suppose that $\val_p(f) \in V$  and that $\mathbf c= (c_1, ..., c_n)\in (\N^*) ^ n$ 
is an evaluator  for $V$.  
 Then $f(p +  ( y^{c_1}, \ldots ,  y^{c_{n}}) )$ is not identically equal to zero 
and its order  at $y=0$ equals 
$\prodscal {\mathbf c} {\val_p (f)}$. 
\end{prop}

\begin{proof}
Write 
$f(x) = \sum_{\mathbf v \in \N^n} a_{\mathbf v} (x-p)^{\mathbf v}$ and denote $\Lambda = \{ \mathbf v; a_{\mathbf v} \ne 0\}$.  
Then 
\begin{align*}
f(p +  ( y^{c_1}, \ldots ,  y^{c_{n}}) )= \sum _{\mathbf v\in \Lambda} a_v y^{\prodscal {\mathbf c} 
{\mathbf v}} = 
a_{v_p(f)} y ^{\prodscal {\mathbf c} {\val_p (f)}} + \sum_{\mathbf u>_{lex} v_p(f)} a_u y ^{\prodscal {\mathbf c} {\mathbf u} }
\end{align*} 
and the proposition follows from Lemma \ref{testexponents}.
\end{proof}

\begin{prop}
Let $f\in K [x_1,...,x_n]$, $f\ne 0$, $\alpha  \in K^{n-1}$.  
Let $\mathbf v= (v_1, \ldots , v_{n-1})$ be the Lazard valuation of $f$ on $\alpha$. 
With $V' \subset \N^{n-1}$ a non-empty family,
suppose that  
$\mathbf v \in V'$ and that   ${\mathbf c'}= (c_1, ..., c_{n-1} )$ 
is an evaluator  for $V'$.   
Then, the following formula holds:
\begin{align}
f(\alpha +  ( y^{c_1}, \ldots ,  y^{c_{n-1}}),x_n )= 
y^{\prodscal {\mathbf c'} {\mathbf v}} (f_{\alpha} ^{\mathbf v} (x_n) + y R(y,x_n) ), 
\end{align} 
with $R\in K[y,x_n]$. 
\end{prop}

\begin{proof}
This follows from the definition of Lazard evaluation on $\alpha$.   
\end{proof}

We have also a parametrised version of the above result.  

\begin{prop}\label{parameterized}
Let $f\in K [x_1,...,x_n]$, $f\ne 0$, $S \subset K^{n-1}$.  
Let $\mathbf v= (v_1, \ldots , v_{n-1})$ be the minimal value 
of Lazard  valuation of $f$ on $\alpha\in S$.  
With $V' \subset \N^{n-1}$ a non-empty family,
suppose that  
$\mathbf v \in V'$ and that   ${\mathbf c'}= (c_1, ..., c_{n-1} )$ 
is an evaluator  for $V'$.   
Then, the following formula holds for all $\alpha \in S$:
\begin{align}\label{first}
f(\alpha +  ( y^{c_1}, \ldots ,  y^{c_{n-1}}),x_n )= 
y^{\prodscal {\mathbf c'} {\mathbf v}} ( f_{\alpha}^ {\mathbf v} ( x_n) + y R(\alpha, y, x_n) ),
\end{align} 
with $f_{\alpha} ^{\mathbf v} (x_n)  \in K[\alpha , x_n]$, $R\in K[\alpha, y, x_n]$. 
\end{prop}

\begin{proof}
This follows from the fact that  the coefficient $f_{\alpha} ^{u}$  in the expansion \eqref{evaluationexpansion}
 is a polynomial in $(\alpha, x_n)$. (If $K$ is of characteristic zero then it is equal to 
  $ \displaystyle {
\frac 1 {u_1! \cdots u_{n-1}!} \frac { \partial^{u_1+ \cdots + u_{n-1}} f }{\partial x_1^{u_1} \cdots \partial x_{n-1}^{u_{n-1}}}}$.)  By assumption on $\mathbf v$,   if  $u <_{lex} \mathbf v$ then $f_{\alpha} ^{u}$  vanishes identically for $\alpha\in S$.  
\end{proof}

\subsection{Proof of Theorem \ref{mainthm}}
The proof is based on the propositions stated in the previous subsection
(especially Proposition \ref{parameterized}) and the Puiseux with parameter theorem 
that we recalled in the previous section.  

Recall that, in the statement of the theorem to be proved, $x$ denotes the
$(n-1)$-tuple $(x_1, \ldots, x_{n-1})$. Write
\begin{align}\label{polynomial}
f(x,x_n) = a_d(x) x_n^d + a_{d-1}(x) x_n^ {d-1} + \cdots + a_0(x).
\end{align}
(Then $a_d(x) = l(x)$ and $a_0(x) = t(x)$.)
We fix positive integers $c_1, ..., c_{n-1}$ that satisfy the following properties.  
Firstly, we want ${\mathbf c} = (c_1, ..., c_{n-1})$ to be an evaluator for
the set $V_S$ of Lazard valuations of $f$ on $p \in S$. 
Secondly we require that ${\mathbf c}$  should be an evaluator
for $V_{g_i}= \{\val _p (g_i): p\in S\}$, for $1 \le i \le 3$, where 
$g_1(x) = D(x)$, $g_2(x) = a_d(x)$ and $g_3(x) = a_0(x)$.
Since $V_S$ and the $V_{g_i}$ are  finite such  $c_1, ..., c_{n-1}$ exist.  
Later in the course of the proof we may multiply all $c_i$ 
by a positive integer $N$.  
Then clearly the vector $(Nc_1, ..., Nc_{n-1})$ still is an
evaluator for $V_S$ and the $V_{g_i}$. 
We shall use test monomial curves, as in Proposition \ref{order}, in the proof. 
Indeed we translate
the assumed Lazard invariance of $D$, $a_d$, and $a_0$ for $p\in S$ near a fixed point
$p_0$ of $S$ into the invariance of the order of these polynomials at $y = 0$
along a suitable monomial curve parametrized by $y$ and passing through $p$,
for $p \in S$ near $p_0$.
This sets the stage for application of Lemma \ref{alonghypersurface} followed
by the Puiseux with parameter theorem. The details follow.

By the assumptions of Theorem \ref{mainthm},
$g_i$ is valuation-invariant in $S$, for $1 \le i \le 3$.
Consider  $\psi : S\times \R \to \R^{n-1}$ defined by 
\begin{align}\label{psi}
\psi (p,y) =  p +  ( y^{c_1}, \ldots ,  y^{c_{n-1}})  
\end{align}
and $g_i(\psi(p,y))$, $1 \le i \le 3$.
By Proposition \ref{order}, for $i$ and $p$ fixed,  
$g_i(\psi(p,y))$, as a function of $y\in \R$,  
has order $k_i := \prodscal {\mathbf c} {\val_p(g_i)}$ at $y=0$.
Since $g_i$ is valuation-invariant in $S$, this order is independent of $p\in S$.  
Therefore, by Lemma \ref{alonghypersurface},
applied with respect to suitable local coordinates on $S \times \R$
near $(p_0,0)$, 
$g_i(\psi(p,y))$ as an analytic function is divisible by $y^{k_i}$ 
with the quotient analytic and non-vanishing in a neighbourhood $U_S \times U'$
of $(p_0,0)$ in $S \times \R$. 
Let 
\begin{align*}
f_\psi (p,y,z) = f(\psi (p,y), z) = a_d(\psi (p,y)) z^d + a_{d-1}(\psi (p,y)) z^ {d-1} + \cdots + a_0(\psi (p,y)). 
\end{align*}
Since $a_d(\psi (p, y))$ is not vanishing identically, 
the degree in $z$ of $f_\psi $ equals $d$ and 
the discriminant $D_\psi (p,y)$ of $f_\psi$ equals $D(\psi (p,y))$.  
Therefore we may apply to 
$f_\psi $, localised at $(p_0,0) \in S\times \R$,  and after complexification,  
the Puiseux with parameter theorem in the form given by Corollary \ref{Puiseuxfinal}.  
Then we use the conclusion of Corollary \ref{Puiseuxfinal} 
to show the Lazard delineability of $f$ over a neighbourhood of $p_0$ in $S$.  
Now we present in detail this argument.  

Where $k$ denotes the dimension of $S$,
we choose local coordinates $(\hat{x}, y) = (\hat{x}_1, \ldots, \hat{x}_k,y)$
on $S \times \R$ near $(p_0,0)$.
We denote by $\hat{f}_{\psi}(\hat{x},y,z)$ the polynomial
$f_{\psi}$ expressed in these coordinates,
and by $\hat{a}_i$ and $\hat{D}_{\psi}$ its $i$th coefficient and discriminant,
respectively. 
Then the $\hat{a}_i$ are analytic in a neighbourhood of the origin in
$\R^{k+1}$. We denote by $\hat{a}_i$ also the complexification of $\hat{a}_i$,
that is, the unique complex analytic extension of $\hat{a}_i$ to a neighbourhood
of the origin in $\C^{k+1}$. 

Hence, in a suitable polydisk $U_{\varepsilon,r}$ in $\C^{k+1}$,
$\hat{f}_{\psi}$ (and $\hat{a}_d$, $\hat{a}_0$, $\hat{D}_{\psi}$) satisfy
the hypotheses of Corollary \ref{Puiseuxfinal} (as was observed above).
Therefore, by this Corollary, we have
\begin{align}
t^{dm_0} \hat{f}_\psi(\hat{x}, t^N, z) =
\hat{a}_d(\hat{x},t^N) \prod_{i} (t^{m_0} z - t^{m_i} \hat{u}_i (\hat{x},t))  
\end{align}
in $U_{\varepsilon,r^{1/N}}$,
for suitable integers $N > 0$ and $m_i \ge 0$, 
and suitable functions $\hat{u}_i(\hat{x}, t)$
analytic and nonvanishing in the specified polydisk.

Now we may write $\hat{a}_d(\hat{x},t^N)= t^{m_0}  \hat{\tilde a}_d(\hat{x},t^N)$, with
$\hat{\tilde a}_d(\hat{x},t^N)$ nonvanishing in the specified polydisk,
as in the proof of the Corollary just cited.  Hence, in our original coordinates, we have
\begin{align}\label{decompfpsi}
t^{(d-1)m_0} f_\psi(p, t^N, z) =  {\tilde a}_d(p,t^N) \prod_{i} (t^{m_0} z - t^{m_i} u_i (p,t)) 
\end{align}
in $U_S \times U'$, after suitable refinement of this neighbourhood if necessary,
where the functions ${\tilde a}_d$, $u_i$ are analytic and nonvanishing in $U_S \times U'$.
Moreover, again by the conclusion of Corollary \ref{Puiseuxfinal},  each difference  
$t^{m_i} u_i (p,t) - t^{m_j} u_j (p,t)$ for $i\ne j$ equals a power of $t$ 
times a nowhere vanishing function.  
Let $\mathbf v$ be the minimum of valuations of $f$ on $p$ for $p\in U_S$.  By \eqref{first} 
\begin{align}\label{firstconclusion}
f_\psi(p, t^ N, z)  = t^{N \prodscal {\mathbf c} {\mathbf v}} (
f_{\alpha}^ {\mathbf v} (z)  
+ t^N R(p, t^N, z) ). 
\end{align} 
By comparing  \eqref{firstconclusion} and \eqref{decompfpsi} we see that 
$N \prodscal {\mathbf c} {\mathbf v} = m_0(1-d)+ \sum_i\max\{m_0,m_i\}$.  
(The nonvanishing of the $u_i$ is used here.)
Therefore the Lazard valuation of $f$ on $p$ is independent of $p\in U_S$, 
and the Lazard evaluation of $f$ at $p$ equals 
\begin{align}
f_p^{\mathbf v} (z) = 
{\tilde a}_d(p,0) z^\nu  \prod_{m_i= m_0} (z-u_i (p,0))  \prod_{m_i <  m_0}  u_i (p,0),  
\end{align}
where $\nu$ equals the number of $i$ with $m_i>m_0$.  
By Corollary \ref{Puiseuxfinal}, each difference 
$u_i (p,0) - u_j (p,0)$, for $i,j$ such that $m_i=m_j=m_0$, 
is either non-vanishing for all $p \in U_S$ or identically equal to zero in $U_S$.  
Hence, by the last part of Corollary \ref{Puiseuxfinal} 
(concerning the complex conjugation invariance of the $u_i$), 
each $u_i(p,0)$ is 
either real for all $p$, if  $u_i(p,0)\equiv \overline u _i (p,0)$, 
or nonreal for all $p$, otherwise.
Thus we may take as the $\theta_j (p)$ of (2) of Definition \ref{delineability}
those $u_i(p,0)$ that are real, and complete 
them by $\theta (p) \equiv 0$ if $\nu$ of the equation above is strictly positive.  
This shows that $f$ is Lazard analytic delineable on $U_S$. 
The Lazard analytic delineability of $f$ on the whole of $S$ 
then follows from the connectedness of $S$.

It remains to address the proof of the second conclusion of Theorem 5.1,
which asserts the Lazard analytic delineability of $f^* = x_n f$ on $S$.
Recall that the functions $u_i(p,0)$ which determine part of 
the variety of $f_p^{\mathbf v}(z)$ for $p$ in
$S$ near $p_0$ vanish nowhere near $p_0$. If $\nu > 0$
the remaining part of the variety of $f_p^{\mathbf v}(z)$ is given by the function $z \equiv 0$.
Therefore each real root function $\theta$ of $f$ on $S$ satisfies the following:
either $\theta(p) < 0$ for all $p$,
or $\theta(p) = 0$ for all $p$,
or $\theta(p) > 0$ for all $p$.
Therefore $f^* = x_n f$ is Lazard analytic delineable on $S$.
\qed

\bigskip

An analogue of our main theorem, with slightly different hypotheses but the
same conclusion, is true. The hypotheses of this analogue are that the discriminant
$D(x)$ is valuation
invariant in $S$ and the leading coefficient $l(x)$ is nonvanishing in $S$. 
No assumption about the trailing coefficient
$t(x)$ is needed. The proof of this analogue is similar to 
(and simpler than) the proof
of our main theorem presented above, except that Theorem 4.1 is used
instead of its corollary, after dividing through by the leading coefficient
of $f_{\psi}$.
This analogue could be used to enhance the practical efficiency of Algorithm 1
in cases where some polynomial elements of the input set $A$ have leading coefficients
which vanish nowhere in $\R^{n-1}$.
As the following examples show we cannot drop the assumption about $t(x)$ 
in general.   

\begin{exmp}
With $(x_1,x_2,x_3) = (x,y,z)$,
consider $f(x,y,z) = y^2z^2 - y(2x+y) z +x(x+y)$ and let $S$ be the $x$-axis. 
Then $D(x,y) =  y^4$, $a_2(x,y) = y^2$ and 
hence both $D$ and $a_2$ are valuation invariant on $S$. 
Nevertheless $f$ is not Lazard delineable on $S$.  
Indeed, $f$ vanishes identically over $x=y=0$ but not over the generic point of $S$.  
An even simpler example is furnished by $f(x,y,z) = yz - x$, with the same $S$.
\end{exmp}

We illustrate the construction in the proof using two examples. For both examples
we take $n = 4$ and $(x_1,x_2,x_3,x_4) = (x,y,z,w)$.
\begin{exmp}
Let $f(x,y,z,w) = yw^ 2 + x w - y z^2$ and let $S\subset \R^3$ be the positive direction of the 
$z$-axis (not including the origin). 
The discriminant $D(x,y,z) =  x^2  + 4y^2z^2 $ vanishes identically on $S$.  
The valuations of $D$, the leading coefficient $a_2=y$, and the trailing coefficient 
$a_0= -y z^2$, at $(0,0,z)$, $z > 0$,  are equal to 
$(0,2,0)$, $(0,1,0)$, and $(0,1,0)$ respectively.  (At the origin they are 
$(0,2,2)$, $(0,1,0)$, and $(0,1,2)$, so we do not include the origin in $S$.)

In order to detect the decomposition of  $S\times \R$ 
into the valuation (of $f$) invariant subsets we  
may take $c_1 = 3, c_2=c_3=1$, 
so that the condition  \eqref{conditiononci} is satisfied for the sets  
$V_g= \{v_p (g) : p\in S\}$, where $g$ equals $D$, $a_2$, and $a_0$, and for 
the set $V_S$ of Lazard valuations of $f$ on $p \in S$.  In this case it means that \eqref{conditiononci} is satisfied for 
$(0,1,0)$ and $(0,2,0)$.  
The function $f_\psi$ is then given by
$$
f_\psi (p,s,w) = sw^2 + s^3 w - s (z+s)^2 = s (w^2 + s^2 w - (z+s)^2),
$$
where $p=(0,0,z)$. (Note that we use the symbol $s$ to denote the second argument
of $f_\psi$, since $y$ already denotes $x_2$.)
This case is particularly simple since by dividing by $s$ and then setting $s=0$ 
we obtain a polynomial of the same degree as $f$. 
This polynomial $w^2 -z^2$ equals  the Lazard evaluation polynomial 
$f_p(w)$, $p=(0,0,z)$, and its zeros  give the decomposition of $S\times \R$ 
into the valuation invariant subsets: 
$v_{(p,w)}(f) = (0,1,0,1)$ if  $w^2=z^2 > 0$ and $v_{(p,w)}(f) = (0,1,0,0)$ if $w^2 \ne z^2$, $z > 0$. 
(If $w=z=0$ then $v_{(p,w)}(f) = (0,1,0,2)$.)

A similar but more complicated example is $f(x,y,z,w) = xw^ 2 + y z w - x$ with the same $S$ and 
similar discriminant  $D(x,y,z) =  y^2z^2 + 4x^2 $.   Then the Lazard evaluation polynomial 
$f_p(w)=zw$ for $p=(0,0,z)$, $z > 0$, (here $z$ is treated as a constant), 
is of degree strictly smaller than the degree of $f$.  
We may take again $c_1 = 3, c_2=c_3=1$.  
The function $f_\psi$ (again with $s$ denoting
the second argument) is given by
$$
f_\psi (p,s,w) = s^3w^2 + s w(z+s) - s^3  = s (s^2 w^2 +  w (z+s)-s^2).
$$
The polynomial in parentheses has two roots for $s\ne 0$.  One of them tends to $0$ as 
$s$ tends to $0$ and the other one tends to infinity. 
In the formula \eqref{decompfpsi}, after dividing through by $t^{(d-1)m_0}$,  
this latter root of $f_\psi(p,t,w)$ has a strictly negative exponent $m_i - m_0$.  
The exponent associated to the first root is strictly positive.
\end{exmp}

\subsection{Derivation of Lazard's main claim from Theorem 5.1}
We show how the special form of Lazard's main claim described just after the
statement of Theorem 5.1 is deduced from Theorem 5.1.
We essentially adapt the proof of Theorem 3.1 of \cite{McCallum:88} in the following way.
Let $x$ denote $(x_1, \dots, x_{n-1})$.
Let $A = \{f_1, f_2, \ldots, f_m\}$ be a finite irreducible basis in $R_n$,
where $n \ge 2$, and let $S$ be a connected analytic submanifold of
$\mathbb{R}^{n-1}$. Suppose that each element of $P_L(A)$ is valuation-invariant in $S$.
If $\pm x_n \notin A$, put $f = f_1 f_2 \cdots f_m$.
Otherwise put $f = f_1 f_2 \cdots f_m / x_n$.
Then $f$ is (up to sign) 
the product of those elements of $A$ whose trailing coefficients are nonzero.
Hence the trailing coefficient $t(x)$ of $f$ is nonzero, and is valuation-invariant
in $S$, by Proposition \ref{valinvproduct}.
Similarly, the leading coefficient $l(x)$ of $f$ is valuation-invariant in $S$.
Using the well-known expression for the discriminant of a polynomial product
and Proposition 3.4,
we see also that the discriminant $D(x)$ of $f$ is nonzero and valuation-invariant in $S$.
Hence, by the first conclusion of Theorem 5.1, $f$ is Lazard analytic delineable on $S$.
If $\pm x_n \notin A$, we may use this property
to deduce the desired conclusions of Lazard's main claim.
If $\pm x_n \in A$, then by the second conclusion of Theorem 5.1,
$f^* = x_n f = f_1 f_2 \cdots f_m$ is Lazard analytic delineable on $S$,
and the desired conclusions of Lazard's main claim follow from this property of $f^*$.

\subsection{Proof of Lazard's original claim}
In this subsection we provide a proof of the {\em original} claim of Lazard,
as stated in Section 2. We first state a corollary of Theorem 5.1:

\begin{cor}\label{mainLazardthm}
Suppose  $f(x, x_n) \in \mathbb{R}[x, x_n]$ 
satisfies the assumption of Theorem \ref{mainthm}.  
Let $S$ be a connected subset of $\mathbb{R}^{n-1}$
in which $D$, $l$ and $t$ are all valuation-invariant.
Then $f$ is Lazard delineable on $S$ and is valuation invariant in every
section and sector of $f$ over $S$.
\end{cor}

\begin{proof}
This corollary follows from Theorem \ref{mainthm} by standard arguments of semialgebraic geometry:  
stratifications and the curve selection lemma.  

First we note that there is a finite semialgebraic stratification $\sqcup_i {T_i} = \R^{n-1}$ such 
that  $D$, $l$ and $t$ are all valuation-invariant on each stratum.  Here, by definition of stratification, each $T_i$ is a connected locally closed semialgebraic subset and an analytic submanifold of $ \R^{n-1}$ and any two strata satisfy the frontier condition: if $T_i\cap \overline T_j \ne \emptyset$ then $T_i\subset \overline T_j$.  
The existence of such a stratification follows from general theory of semialgebraic sets, see e.g. Proposition 
9.1.8 of \cite{BCR:98} or Proposition 2.5.1 of \cite{BR:90}.  

Now for  $ S \subset \mathbb{R}^{n-1}$  
as in the assumption of corollary we consider $S_1 = \sqcup_{i\in \Lambda} {T_i} $, 
the union of all strata intersecting $S$.   
Clearly $S_1$ is connected, $D$, $l$ and $t$ are all valuation-invariant in $S_1$, 
and Lazard delineability on $S_1$ implies Lazard delineability on $S$.  
Thus we may replace $S$ by $S_1$.  

In order to show Lazard delineability on a connected union of strata it suffices to show it on 
$S$ of the form $S = T_i\cup  T_j$ with $T_i\subset \overline T_j$.  Thus suppose that $S = T_i\cup  T_j$  
and that 
$D$, $l$ and $t$ are valuation-invariant on $S $.  We show that  $f$ is delineable on  
$S$.  For this it is enough to show that   the Lazard valuation on $\alpha $ is 
independent of $\alpha \in S$, that the number of sections of $f_\alpha$ over $T_i$ and $T_j$ coincide, 
that these sections are given by functions continuous on whole $S$, and finally that their multiplicities 
as roots of $f_\alpha$ are the same on $T_i$ and $T_j$.  
By standard arguments based on the curve selection lemma, see e.g. Proposition 8.1.13 of \cite{BCR:98},
it suffices to show all these claims over $p([0,\varepsilon))$, 
where  $p(\tau) : (-\varepsilon,\varepsilon) \to S$ is 
an arbitrary real analytic curve such that $p(0) \in T_i$ and $p(\tau) \in T_j$ for $\tau> 0$. 
By replacing $\tau$ by $\tau^2$ we may assume that $p(\tau) \in T_j$ for $\tau\ne 0$.  
Then we  follow  the main steps of the proof of Theorem \ref{mainthm}.  
Let $\psi:(-\varepsilon,\varepsilon)  \times \R \to \R^{n-1}$ be defined by 
\begin{align*}
\psi (\tau,y) = p(\tau) +  ( y^{c_1}, \ldots ,  y^{c_{n-1}})  ,
\end{align*}
where $(c_1, ..., c_{n-1})$  is an evaluator for the sets $V_g= \{\val _p (g): p\in S\}$, with
$g(x) = D(x), l(x) \text{and } t(x)$, and  an evaluator for the set $V_S$ of Lazard valuations of $f$
on $p \in S$.
Since $D$, $l$ and $t$ are valuation-invariant on 
the image of $p(\tau)$, $f_\psi (\tau,y,z) = f(\psi (\tau,y), z) $ 
satisfies the assumptions of Puiseux with parameter theorem, 
Corollary  \ref{Puiseuxfinal}.   Then the proof of Theorem \ref{mainthm}  shows  
that $f$ is Lazard delineable on the image of $p(\tau)$, 
and hence, by the curve selection lemma, on $S$.    
\end{proof}

Lazard's original claim follows from the above corollary
by analogy with the argument presented 
in Subsection 5.3 above.

\section{Conclusion}
We first summarise the work reported herein.
We presented the results of our
investigation of Lazard's proposed CAD method,
including both his proposed projection and valuation. 
In \cite{McCallum_Hong:16} we already found that
Lazard's projection is valid for CAD construction for well-oriented polynomial sets. 
In the present paper (Section 5) Lazard's main claim is proved using his valuation.
A consequence of this result is that Lazard's CAD method is valid, with no
well-orientedness restriction.

There are immediate consequences of our main result for certain related
works on projection in CAD.
For example, as similarly mentioned in \cite{McCallum_Hong:16} (Section 4),
our main result could be readily adapted to obtain an analogue of the theorem
of \cite{McCallum:99} concerning the reduction of projection sets
in CAD-based quantifier elimination in the presence of equational constraints.
A further significant potential benefit of Lazard's method 
is that it may permit {\em greater} simplifications and improvements to projection
for such problems. Indeed, the newly validated approach which
underpins improved projection using valuation invariance may yield better results
concerning the use of so-called {\em propagated} constraints for such
problems \cite{McCallum:01}.

It is natural to ask how Lazard's CAD method compares with other reduced projection
CAD algorithms with respect to efficiency and other criteria.
As mentioned, Lazard's method is more general than those of McCallum 
\cite{McCallum:84, McCallum:88, McCallum:98} and Brown \cite{Brown:01}
in the sense that the latter algorithms fail for non-well-oriented input sets $A$.
The algorithm of \cite{McCallum:84} was subjected to a theoretical computing
time analysis which was broadly based upon that provided by Collins
\cite{Collins:75} for his original CAD algorithm.
The overall conclusion was that the method of \cite{McCallum:84} remains of
doubly exponential worst case time complexity in the number of variables $n$,
as is the case for Collins' original CAD, 
though the double exponent in the computing time bound is reduced.
Still, for every fixed $n$, both algorithms have a polynomial worst case computing
time bound. (Interested readers are referred to the very recent paper
\cite{BDEMW:16} which contains an improved exposition of the complexity analysis of
\cite{McCallum:84}.) To our knowledge, no computing time analysis of
Brown's method, nor that of Lazard, has yet been published.
Nonetheless, given that both sets $P_{BM}(B)$ and $P_L(B)$ contain the discriminants,
and resultants of pairs of distinct elements, of an irreducible basis $B$,
it is likely that the computing times of the methods of both Brown and Lazard
remain doubly exponential in $n$, perhaps with slight improvements to the double
exponent, relative to McCallum's method.

Further work could usefully be done in a number of directions.
As mentioned and elaborated in \cite{McCallum_Hong:16},
it will be interesting to compare experimentally the Brown-McCallum
projection \cite{Brown:01} with the Lazard projection.
As mentioned above, it would be worthwhile to try to extend the theory 
of equational constraints, especially the use of propagated constraints,
with the Lazard projection. Similarly, re-examination of the theory
of bi-equational constraints \cite{Brown_McCallum:05, Brown_McCallum:09} 
in the context of Lazard's projection and valuation
may be fruitful. Re-visiting ideas for practical improvements
suggested in \cite{Lazard:94} could be beneficial.

Another inportant direction will be to understand the topological and geometric structure 
of the output of the CAD algorithm.  This concerns all the methods, not only the  Lazard one presented in this paper, see \cite{Lazard:10}.  
To study it will be interesting to use the new ideas proposed in a recent paper on Zariski equisingularity and stratifications  \cite{PP15}.  

\section{Appendix.  Proof of Puiseux with parameter theorem}

In this section, for the reader's convenience, we present a concise
proof of Theorem \ref{PuiseuxTheorem}. This proof is based on 
the classical theory of complex analytic functions and uses a parametrized version of the Riemann removable singularity theorem.  The proof we present below is due to  \L ojasiewicz and Paw\l ucki, see \cite{Pawlucki:84}.  
For a similar approach, with slightly different details,  see  \cite{PR:12} Proposition 2.1. 

Let us first recall the basic notation: 
$U_{ \varepsilon, r} =  U_{\varepsilon} \times U_r  $, where 
$ U_\varepsilon = \{x = (x_1, \ldots, x_k) \in \C^k:  |x_i|< \varepsilon_i, \forall i\}$, 
$U_r= \{y\in \C: |y|<r\}$. 
We also denote the punctured disc $U_r \setminus \{0\}$ by  $U^*_r $.

\begin{proof}[Proof of Theorem \ref{PuiseuxTheorem}]
Consider the polynomial in $z$, 
$$P(x,w,z):= f(x, e^{2\pi i w},  z),$$
 whose coefficients $a_i( x, e^{2\pi i w})$ 
are analytic on $U_\varepsilon \times H$, 
where $H = \{w: 2\pi  {\rm Im} (w) > - \ln r\}$.  By assumption the discriminant 
$D_P(x,w) = D _f(x, e^{2\pi i w})$ 
does not vanish on  $U_{\varepsilon}\times H$ 
and hence $P$ admits  global complex analytic roots 
 $$\tilde \xi_1(x, w), \cdots, \tilde\xi_d(x, w).$$
(If  $D_P(x,w)$ is nonzero then the equation $P= \partial P/\partial z=0$ 
has no solution.  
Therefore the local solutions $\tilde \xi (x,w)$ of $P=0$ are analytic by the implicit
function theorem.  
They are well-defined global analytic functions 
because $U_{\varepsilon}\times H$ is contractible.)
 
 The coefficients of $P$ are periodic: $P(x,w+1,z) = P(x,w,z)$.  Hence for each root $\tilde \xi_i(x,w)$ there is another 
 root  $\tilde \xi_{\varphi(i) }(x,w)$ such that $ \tilde \xi_{i }(x,w+1)= \tilde \xi_{\varphi(i) }(x,w)$.   
 The map $\varphi : \{1, \ldots, d\} \to \{1, \ldots, d\}$ is a permutation and hence 
 $\varphi ^{d!}= id$.  Therefore, for $N=d!$, 
 \begin{align}\label{periodic}
 \tilde \xi_{i }(x,w+N)= \tilde  \xi_{i}(x,w). 
 \end{align}
Thus there are analytic functions 
$\xi_i (x,t):U_{\varepsilon} \times U^*_{r^{1/N}} \to \C$ such that 
$\tilde  \xi_{i}(x,w) = \xi_i (x,e^{2\pi i w/N})$.  
Since  $\xi_i (x,t)$ are roots of $f(x,t^N,z) $ they 
are bounded on  $U_{\varepsilon_0} \times U^*_{r_0^{1/N}} $, 
for every $\varepsilon_0< \varepsilon$, $r_0<r$.  
Hence they extend to functions analytic on  $U_{\varepsilon} \times U_{r^{1/N}}$ 
by Riemann's theorem on removable singularities, cf.  \cite{GRossi:09} Theorem 3, p. 19.
\end{proof}

\section*{Acknowledgements}

We acknowledge grants which supported visits to the University of Sydney by the
second named author in 2015 and 2016 : 
Sydney University BSG, IRMA Project ID: 176623 and  ANR project STAAVF (ANR-2011 BS01 009).
We are grateful to Hoon Hong for kindly granting us permission to include
some basic material from the technical report \cite{McCallum_Hong:15}
in Section 3.

\end{document}